\documentclass[12pt]{amsart}
\usepackage[margin=1.25in]{geometry} 
\usepackage{enumitem}
\usepackage{mathrsfs}
\usepackage{hyperref}
\usepackage{comment}
\usepackage{amsthm,amsrefs,amssymb,amsmath}
\numberwithin{equation}{section}
\usepackage{todonotes}
\usepackage{booktabs}
\usepackage{placeins}
\usepackage{amscd,amsmath,amssymb,amsthm,amsfonts,epsfig,graphics}
\usepackage{setspace}
\usepackage{titlesec}

\newcommand{\Z}{\mathbb{Z}}
\newcommand{\N}{\mathbb{N}}
\newcommand{\C}{\mathbb{C}}
\newcommand{\Q}{\mathbb{Q}}

\newcommand{\bd}{\begin{description}}
\newcommand{\ed}{\end{description}}

\newtheorem{theorem}{Theorem}[section]
\newtheorem{cor}[theorem]{Corollary}
\newtheorem{lemma}[theorem]{Lemma}
\newtheorem{prop}[theorem]{Proposition}

\theoremstyle{definition}
\newtheorem{definition}[equation]{Definition}
\newtheorem{definition*}[theorem]{Definition}

\theoremstyle{remark}

\DeclareMathOperator{\spec}{Spec}
\DeclareMathOperator{\sing}{Sing}

\newtheoremstyle{dotless}{}{}{\itshape}{}{\bfseries}{}{ }{}
\theoremstyle{dotless}

\numberwithin{equation}{section}
\titleformat{\section}
  {\centering\normalfont\scshape}
  {\thesection}{1em}{}
\title{\scshape \normalsize Uncountable $n$-dimensional Excellent Regular Local Rings with Countable Spectra}

\date{}
\begin{document}
\onehalfspacing

\author{S. Loepp and A. Michaelsen}
%\address{Department of Mathematics, Williams College, 33 Stetson, Williamstown, MA 01267}
%\email{anm1@williams.edu}

\pagenumbering{arabic}
\maketitle

\begin{abstract}
We prove that, for any $n\geq 0$, there exists an uncountable, $n$-dimensional, excellent, regular local ring with countable spectrum. 
\end{abstract}

\section{Introduction}

A remarkable result by M. Hochster in \cite{Hochster} gave exact conditions under which a partially ordered set (poset) could be realized as the spectrum of a commutative ring. Little is known on this question, however, when we place further restrictions on the resulting ring. For instance, the question is open when we require the ring to be Noetherian. Previous work done by R. Wiegand, S. Wiegand, and C. Colbert, among others, has provided specific examples of posets that can be realized as the spectra of Noetherian rings \cites{Colbert, Wiegand}. For instance, C. Colbert recently proved, for any $n\geq 2$, the existence of an uncountable, $n$-dimensional, Noetherian ring with a countable spectrum. While this is clear for dimensions 0 and 1, it was unknown in higher dimensions until Colbert proved his result. 
In this paper, we prove a similar result to Colbert's but with stronger conditions on the ring. In particular, our main result is as follows: 

\begin{theorem}
For any $n\geq 0$, there exists an uncountable, $n$-dimensional, excellent, regular local ring with a countable spectrum. 
\end{theorem}

In the case of $n=0$ and $n=1$, $\C$ and $\C[[x]]$, respectively, are examples of such  rings. In the case that $n\geq 2$, we prove the existence of such a ring constructively. The construction takes place between the polynomial ring $\Q[x_1,\ldots, x_n]$ and the corresponding power series ring $\Q[[x_1,\ldots, x_n]]$ and consists of two major steps. In the first step, we construct a local (Noetherian) countable base ring, $S$, with completion  $\Q[[x_1,\ldots, x_n]]$ such that $S$ is both excellent and has the property that it contains (up to units) every prime element in the power series ring that divides any element in $S$. In the second step, we algorithmically adjoin uncountably elements to this ring $S$ so that every ideal of the resulting ring, $B$, is extended from $S$. We use this to show that $B$ has completion  $\Q[[x_1,\ldots, x_n]]$ and is an excellent regular local ring. 

In section 2 we present preliminaries. In section 3 we construct the base ring, $S$. In section 4 we introduce some useful definitions for the remainder of the construction. In section 5 we construct the final ring, $B$, and then prove that $B$ is an uncountable, $n$-dimensional, excellent, regular local ring with a countable spectrum in section 6.

\section{Preliminaries}\label{preliminaries}

For the rest of this paper assume $n\geq 2$, and define $R_0=\Q[x_1, \ldots, x_n]$, $T=\Q[[x_1,\ldots, x_n]]$ and $M=(x_1, \ldots, x_n)T$ the maximal ideal of $T$. We will use \textit{quasi-local} to refer to a ring with a unique maximal ideal and \textit{local} to refer to a Noetherian quasi-local ring. When $R$ is a local ring, $\widehat{R}$ indicates the completion of $R$ at its maximal ideal $M$.

In this paper we will be constructing many rings with completion $T$. To show that they have this property, we make use of the following proposition. 

\begin{prop}[{\cite[Proposition 1]{Heitmann}}]\label{completion proving machine}
If $(R,R\cap M)$ is a quasi-local subring of a complete local ring $(\mathcal{R},M)$, the map $M\rightarrow \mathcal{R}/M^2$ is onto and $I\mathcal{R}\cap R=IR$ for every finitely generated ideal $I$ of $R$, then $R$ is Noetherian and the natural homomorphism $\widehat{R}\rightarrow \mathcal{R}$ is an isomorphism. 
\end{prop}

Applied to a subring of $T=\Q[[x_1,  \ldots, x_n]]$, this yields the following corollary. 

\begin{cor}\label{alt completion proving machine}
Let $(R,R\cap M)$ be a quasi-local subring of $T$ such that $R_0\subseteq R$, and, assume that, for every finitely generated ideal $I$ of $R$, $IT\cap R=IR$. Then $R$ is Noetherian, $\widehat{R}=T$ and $R$ is a regular local ring (RLR). 
\end{cor}
\begin{proof}
We will  use Proposition \ref{completion proving machine} to show that $R$ has completion $T$. To do this, we will show that the map $R\rightarrow T/M^2$ is onto. Given $t+M^2\in T/M^2$, %we know that $t=a_0+a_1x_1+\cdots + a_nx_n+m$ for $a_i\in \Q$ and $m\in M^2$. Define $s=a_0+a_1x_1+\cdots + a_nx_n$. Notice that $s\in R_0\subseteq R$. 
we know that $t=s+m$ for $s\in R_0$ and $m\in M^2$. 
Then, $t-s=m\in M^2$, so $s+M^2=t+M^2$. Thus the map is onto. Since, for any finitely generated ideal $I$ of $R$ we have that $IT\cap R=IR$, by Proposition \ref{completion proving machine}, $R$ is Noetherian with completion $T$. Furthermore, since $T$ is a RLR, so is $R$.  
\end{proof}

Note that if $(R,R\cap M)$ is a local ring with $\widehat{R}=\mathcal{R}$, then $\mathcal{R}$ is a faithfully flat extension of $R$ and so any finitely generated ideal $I$ of $R$ satisfies $I\mathcal{R}\cap R=IR$. 
The following definitions and lemmas pertain to excellent rings. 
Define, for any $P\in\spec A$, $k(P)=A_P/PA_P$.

\begin{definition}[{\cite[Definition 1.4]{Rotthaus}}]
\label{def of excellent}
Given a local ring $(A,A\cap M)$, $A$ is excellent if
\begin{itemize}
\item[(a)] For all $P\in \spec A$, $\widehat{A}\otimes_A L$ is regular for every finite field extension $L$ of $k(P)$, and 
\item[(b)] $A$ is universally catenary. 
\end{itemize}
\end{definition}

As noted in \cite{Rotthaus}, we can consider only the purely inseparable finite field extensions $L$ of $k(P)$.
The following is a consequence of Theorem 31.6 and the definition of formally equidimensional in \cite[pp. 251]{Matsumura}.

\begin{theorem}\label{universally catenary}
Let $(A,M)$ be a local ring such that its completion, $\widehat{A}$, is equidimensional. Then $A$ is universally catenary. 
\end{theorem}

We will now give sufficient criteria for excellent for rings with completion $T$.

\begin{lemma}\label{alt def of excellent}
Given a local ring $(A,A\cap M)$ with $\widehat{A}=T$, $A$ is excellent if, for every $P\in\spec A$ and for any $Q\in\spec T$ with $Q\cap A=P$, $(T/PT)_{Q}$ is a RLR. 
\end{lemma}
\begin{proof}
We know that $A$ is a local ring, and so we must show both conditions of Definition \ref{def of excellent} hold. By Theorem \ref{universally catenary} $A$ is universally catenary. We must then consider $T\otimes_A L$ for every purely inseparable finite field extension of $k(P)$ for each $P\in\spec A$. Since $\Z \subset A$ and all nonzero integers are units, we have that $\Q\subset k(P)$, and so $k(P)$ has characteristic 0. 
Every finite field extension with characteristic 0 is separable. Since it is sufficient to check only purely inseparable field extensions, this leaves only the trivial field extension so we need only show that $T\otimes_A k(P)$ is regular for every $P\in \spec A$. 
Note that $T\otimes _A k(P)$ localized at $Q\otimes k(P)$ is isomorphic to $(T/PT)_Q$. Hence, it suffices to show that $(T/PT)_Q$ is a RLR. 
\end{proof}

We end with a lemma regarding the structure of $\sing(R)$ of excellent rings. 

\begin{lemma}[{\cite[Corollary 1.6]{Rotthaus}}]\label{singular locus is closed}%cor 1.6 in Rotthaus 
If $R$ is excellent, then $\sing(R)$ is closed in the Zariski topology, i.e. $\sing(R)=V(I)$ for some ideal $I$ of $R$. 
\end{lemma}

\section{Construction of The Base Ring}\label{construction of base ring}
We first construct a countable local excellent ring $(S,S\cap M)$, where $R_0\subseteq S\subseteq T$, $\widehat{S}=T$, and, for any nonzero $s\in S$ with $s\in pT$  for some prime element $p$ of $T$, $pu\in S$ for a unit $u\in T$.

\begin{lemma}\label{add-ideal}
Let $R$ be a countable ring with $R_0\subseteq R\subseteq T$. Then there exists a countable ring $R'$ such that $R\subset R'\subset T$ and, if $I$ is a finitely generated ideal of $R$, then $IT\cap R\subseteq IR'$.
\begin{proof}
Consider the set  
$\Omega = \{(I,c): I\text{ is a finitely generated ideal of }R, c\in IT\cap R\}.$
Since $R$ is countable, so is $\Omega$. Enumerate $\Omega$ with 0 denoting its first element. We will inductively define countable subrings for each $(I,c)\in \Omega$. First let $S_0=R$. Notice that $S_0$ is countable.
Given the $k$th element, $(I,c)$, with the countable ring $S_k$ defined, we will construct $S_{k+1}$. Let $I=(a_1,\ldots, a_\ell)R$ for $a_i\in R$. Then, since $c\in IT\cap R$, we have that $c=a_1t_1+\cdots+a_\ell t_\ell$ for $t_i\in T$. Now let $S_{k+1}=S_k [t_1,\ldots, t_\ell]$. Since $S_k$ is countable, so too is $S_{k+1}$. Furthermore, notice that $c\in IS_{k+1}$.  

Then define $R'=\bigcup_{i=0}^\infty S_i$. We will show that $IT\cap R\subseteq IR'$ for any finitely generated ideal $I$ of $R$. Given a finitely generated ideal $I$ of $R$, and $c\in IT\cap R$, then $(I,c)\in \Omega$. If $(I,c)$ is the $j$th element of $\Omega$, then $c\in IS_{j+1}$. Notice $IS_{j+1}\subseteq IR'$, so then $IT\cap R\subseteq IR'$. 
\end{proof}
\end{lemma}

\begin{lemma}\label{close up ideals}
Let $R$ be a countable ring with $R_0\subseteq R\subseteq T$. Then there exists a countable local ring $(R'',R''\cap M)$ such that, $R\subseteq R''\subseteq T$, and, for every finitely generated ideal $I$ of $R''$, $IT\cap R''=IR''$. Hence, $R''$ is Noetherian, $\widehat{R''}=T$, and $R''$ is a RLR.
\begin{proof}
Let $S_0=R$. Note that $S_0$ is countable. Given the ring $S_i$, let $S_{i+1}$ be the countable ring obtained from Lemma \ref{add-ideal}, such that, if $I$ is a finitely generated ideal of $S_i$, we have $IT\cap S_i\subseteq S_{i+1}$. 

Now, define $R'=\bigcup_{i=0}^\infty S_i$. We will show that, for every finitely generated ideal $I$ of  $R'$, $IT\cap R'=I$. Consider some finitely generated ideal $I$ of $R'$. Clearly $IR'\subseteq IT\cap R'$, so suppose $c\in IT\cap R'$. Let $I=(a_1,\ldots, a_m)R'$. For some $k$, $a_i\in S_k$ for all $i$, and $c\in S_k$. Then $c\in (a_1,\ldots, a_m)T\cap S_k\subseteq (a_1,\ldots, a_m)S_{k+1}\subseteq IR'$. Thus we have that $IT\cap R'=IR'$. 

Then define $R''=R'_{R'\cap M}$; we will show that $R''$ will also have the property that, for any finitely generated ideal $I$ of $R''$, $IT\cap R''=IR''$.
If $I=R''$ this holds. 
If $I$ is a proper finitely generated ideal of $R''$, then $I=\big(\tfrac{a_1}{b_1},\ldots, \tfrac{a_m}{b_m}\big)R''=(a_1,\ldots, a_m)R''$ for $a_i\in R'$. Since $I$ is proper, we know that $a_i\in R'\cap M$. Suppose $x\in IT\cap R''$.  Then since $x\in R''$, $x=ab^{-1}$ for $a\in R'$ and $b\notin R'\cap M$. Since $x=ab^{-1}\in IT$, $xb=a\in IT\cap R'=(a_1,\ldots, a_m)R'\subseteq IR''$. Since $b$ is invertible in $R''$, this means that $x=ab^{-1}\in IR''$, so $IT\cap R''=IR''$, as desired. 
Thus by Corollary \ref{alt completion proving machine}, $R''$ is Noetherian, has completion $T$, and is a RLR.  
\end{proof}
\end{lemma}

Since our base ring, $(S,S\cap M)$, will have completion $T$, by Lemma \ref{alt def of excellent}, to show that $S$ is excellent it is sufficient to show for each $P\in\spec S$ and $Q\in \spec T$ lying over $P$, $(T/PT)_Q$ is a RLR. Given some intermediate ring $R$ with completion $T$, and some $P\in\spec R$, the next lemma shows that there are only countably many $Q\in\spec T$ lying over $P$ in $R$ such that $Q$ is minimal in $\sing(T/PT)$.

\begin{lemma}\label{countable min primes}
Let $(R,R\cap M)$ be a countable local ring with $R_0\subseteq R\subseteq T$ and $\widehat{R}=T$. Then 
$\bigcup\limits_{P\in\spec R}\{Q_j\in\spec T: {Q_j}\in\min I \text{ for $I$ where }\sing(T/PT)=V(I/PT)\}$
is a countable set. 
\begin{proof}
First, since $R$ is countable and Noetherian, $\spec R$ is countable, so it suffices to show that the set is countable with respect to any fixed $P\in\spec R$. Let $P\in \spec R$. Then, since $T$ is a complete local ring, $T/PT$ is excellent. By Lemma \ref{singular locus is closed} $\sing(T/PT)=V(I/PT)$ for some ideal $I$ of $T$.
Then consider the set of minimal prime ideals $Q_j$ of $I$. Since $T$ is Noetherian, this set is finite.
\end{proof}
\end{lemma}

\begin{definition}\label{factors}
We say a subring $R$ of $T$ \textit{contains all its factors}, if, for any nonzero $r\in R$ with $r\in pT$  for some prime element $p$ of $T$, $pu\in R$ for a unit $u\in T$.
\end{definition}

\begin{theorem}\label{ring S}
There exists an excellent, countable, RLR, $(S,S\cap M)$, such that $R_0\subseteq S\subseteq T$,  $\widehat{S}=T$, and $S$ contains all its factors.
\begin{proof}
Let $S_0={R_0}_{R_0\cap M}$. Notice that $S_0$ is local, has completion $T$ and is a RLR.
 We will define an ascending chain of rings recursively. For each $S_i$ we will ensure that it satisfies the criteria of Lemma \ref{countable min primes}, i.e. that $(S_i,S_i\cap M)$ is a countable local ring with $R_0\subseteq S_i\subseteq T$ and $\widehat{S_i}=T$. Notice that, in the base case, $S_0$ satisfies these conditions. 
Assume that $(S_i,S_i\cap M)$ is a countable local ring with $R_0\subseteq S_i\subseteq T$ and $\widehat{S_i}=T$. We will construct $S_{i+1}$ to satisfy these as well. First, for each $s\in S_i$ where $s$ is a nonzero non-unit, since $T$ is a UFD, choose exactly one factorization of $s$ in $T$, $s=p_{s_1}\cdots p_{s_m}$, for prime elements $p_{s_j}$ in $T$. Define $P(S_i)=\bigcup\limits_{s\in S_i} \{p_{s_j}\}$. Since each $s$ adds only  finitely many $p_i$ and $S_i$ is countable, this set is countable. 

Next, by Lemma \ref{countable min primes}, the set
$$ \bigcup_{P\in\spec S_i}\{Q_j: {Q_j}\in\min I \text{ for $I$ where }\sing(T/PT)=V(I/PT)\}$$
is countable. Since $T$ is Noetherian, every $Q_j$ is finitely generated. Choose exactly one generating set, $\{q_{j_1},\ldots, q_{j_\ell}\}\subseteq T$ for each $Q_j$. Define $G(S_i)=\bigcup_{Q_j}\{q_{j_k}\}$. Each $Q_j$ adds only finitely many $q_j$'s so $G(S_i)$ is countable. 

Now given $S_i$, define $S'_i=S_i[G(S_i),P(S_i)]$. Notice since $S_i, G(S_i)$, and $P(S_i)$ are countable, so is $S'_i$. Then let $(S_{i+1},S_{i+1}\cap M)$ be the countable local ring obtained by applying Lemma \ref{close up ideals} to $S'_i$, so $\widehat{S}_{i+1}=T$ and $S_{i+1}$ is a RLR. Notice that $S_{i+1}$ satisfies the conditions of Lemma \ref{countable min primes} as needed.  

Define $S = \bigcup_{i=0}^\infty S_i$. We will show that $S$ is a countable, excellent, RLR such that $R_0\subseteq S\subseteq T$, $\widehat{S}=T$, and $S$ contains all its factors. First, since each $S_i$ is countable, their countable union is countable so $S$ is countable, as desired. Furthermore, since each $S_i$ satisfies $R_0\subseteq S_i\subseteq T$, so does $S$. 
Next, since each  $S_i$ has a unique maximal ideal $S_i\cap M$, $S$ is is quasi-local with maximal ideal $S\cap M$. 
Note that, for every $i$, since $\widehat{S_i}=T$, if $I$ is a finitely generated ideal of $S_i$, then $IT\cap S_i=IS_i$. We will show that this holds for $S$ as well. 

Suppose $I$ is some finitely generated ideal of $S$. It is clear that $IS\subseteq IT\cap S$, so we will show that $IT\cap S\subseteq IS$. Since $I$ is finitely generated, $I=(a_1,\ldots, a_m)S$ for $a_i\in S$. Let $c\in IT\cap S$. Choose $\ell$ such that $a_i,c\in S_\ell$. Then $c\in (a_1,\ldots, a_m)T\cap S_\ell= (a_1,\ldots, a_m)S_{\ell}\subset IS$.  Thus $IT\cap S\subseteq IS$. Since $R_0\subseteq S$,   by Corollary \ref{alt completion proving machine}, $S$ is Noetherian, has completion $T$, and is  a RLR.
 
Next we will show that $S$ contains its own factors. Let $r\in S$ and $r\in pT$ for some prime element $p$ in $T$. Then $r\in S_i$ for some $i$. Thus, $pu\in P(S_i)$ for some unit $u$ in $T$. Since $P(S_i)\subset S_{i+1}\subseteq S$,  we have that $pu\in S_{i+1}\subseteq S$ and we see that $S$ contains its own factors.

Finally, we will show that $S$ is excellent by showing that, for $P\in\spec S$, and $Q\in\spec T$ such that $Q\cap S=P$, $(T/PT)_{Q}$ is a RLR. 
Let $P\in\spec S$ and $Q\in \spec T$ such that $Q\cap S=P$. Suppose for contradiction that $Q/PT\in \sing(T/PT)=V(I/PT)$. Then $Q\supseteq {Q_j}\supseteq I$ for some minimal prime ideal $Q_j\in\spec T$ of $I$. Since ${Q_j}\supseteq I\supseteq PT$, ${Q_j/PT}\in\sing(T/PT)$.
% and $Q_j\cap S=P$. 
%because Q_j\cap S\subseteq Q\cap S=P and P=PT\cap S\subseteq Q_j\cap S, so Q_j\cap S=P
Furthermore, $P=PT\cap S\subseteq Q_j\cap S\subseteq Q\cap S=P$ so $Q_j\cap S=P$. 

 Note that $P$ is finitely generated, so let $P=(p_1,\ldots, p_m)S$. Choose $i$ so that $p_j\in S_i$ for all $j$. Then define $P'=P\cap S_i$ and note that $P'\in\spec S_i$. We will first show that $P'T=PT$. Observe,
 \[P'=P\cap S_i=(PT\cap S)\cap S_i = PT\cap S_i= (p_1,\ldots,p_m)T\cap S_i=(p_1,\ldots, p_m)S_i\]
 so then $P'T = ((p_1,\ldots, p_m)S_i)T=(p_1,\ldots, p_m)T=PT$. 
 Thus $T/PT=T/P'T$, and so since $Q_j/PT\in\sing(T/PT)$, we have ${Q_j/P'T}\in\sing(T/P'T)$. Next we will show that $Q_j=PT$.
 Since $Q_j\cap S=P$, we have $Q_j\cap S_i=P\cap S_i=P'$. So then, $Q_j\in\bigcup_{P'\in\spec S_i}\{Q_j: {Q_j}\in\min I \text{ for $I$ where } V(I/P'T)=\sing(T/P'T)\}$. This means that for some generating set for $Q_j$, $\{q_1,\ldots, q_\ell\}\subseteq T$, we have $\{q_k\}\subset G(S_i)\subseteq S_{i+1}$.
 Thus $Q_j\cap S_{i+1}=(q_1,\ldots, q_\ell)T\cap S_{i+1}=(q_1,\ldots, q_\ell)S_{i+1}$. Since $S\supseteq S_{i+1}$, $Q_j\cap S\supseteq Q_j\cap S_{i+1}=(q_1,\ldots, q_\ell)S_{i+1}$. 
 Recall that $Q_j\cap S=P$, so 
 \[PT=(Q_j\cap S)T\supseteq ((q_1,\ldots, q_\ell)S_{i+1})T=(q_1,\ldots, q_\ell)T=Q_j.\]
 We already know that $Q_j\supseteq PT$, thus $Q_j=PT$. Then  $(T/PT)_{Q_j}=(T/PT)_{PT}$ is a field (and hence a regular local ring). However $Q_j\in\sing(T/PT)$, so $(T/PT)_{Q_j}$ is not a regular local ring, a contradiction. 
 It follows that $Q/PT\notin \sing(T/PT)$ and so $(T/PT)_Q$ is a regular local ring. 
\end{proof}
\end{theorem}

\section{Definitions and Lemmas}\label{definitions and lemmas}

In the previous section we constructed an excellent, countable, RLR, $(S, S\cap M)$, with completion $T$ that contains all its own factors. For the remainder of this paper, $S$ refers to this ring. From $S$, we will construct an uncountable strictly ascending chain of rings such that, for any prime ideal $P\in\spec T$, $P$ intersects nontrivially with each ring if and only if $P$ intersects nontrivially with $S$. 
To accomplish this, we introduce the following definitions inspired by the definition of $W$-subrings by W. Zhu in her senior thesis at Williams College. 

\begin{definition}
A ring $R$ is an $S^\star$-subring of $T$ if  $S\subseteq R\subseteq T$ and, for any $P\in \spec T$, if $P\cap S=(0)$ then $P\cap R=(0)$. 
\end{definition}

\begin{definition} \label{CS^star-subring} 
A $CS^\star$-subring, $R$, of $T$, is a countable quasi-local $S^\star$-subring of $T$ with maximal ideal $R\cap M$. 
\end{definition}

Note that $S$ is an example of a $CS^\star$-subring of $T$. Furthermore, we will show that the union of $S^\star$-subrings of $T$ is an $S^\star$-subring as well. Let $R=\cup_\alpha R_\alpha$ where every $R_\alpha$ is an $S^\star$-subring of $T$. Then clearly $S\subseteq R\subseteq T$. Suppose $P\in\spec T$ such that $P\cap S=(0)$ and let $r\in P\cap R$. Then $r\in R_\alpha$ for some $\alpha$, so $r\in P\cap R_\alpha=(0)$ so $r=0$ and thus $P\cap R=(0)$, as desired. Note if each $R_\alpha$ is quasi-local with maximal ideal $R_\alpha\cap M$, then $R$ will be quasi-local with maximal ideal $R\cap M$. Thus the countable union of $CS^\star$-subrings of $T$ is also a $CS^\star$-subring of $T$ and the uncountable union of $CS^\star$-subrings is a quasi-local $S^\star$-subring. 

\begin{lemma}\label{qTcapS=quS}
Let $R$ be an $S^\star$-subring of $T$, with $q$ a prime element of $T$ such that $qT\cap R\neq (0)$. Then, for some unit $u$ in $T$, $qu\in S$, and $qT\cap S=quS$.  
\end{lemma}
\begin{proof}
Since $R$ is an $S^\star$-subring of $T$ and $qT\in \spec T$ and $qT\cap R\neq (0)$, we have $qT\cap S\neq (0)$. Let $s\in qT\cap S$ with $s\neq 0$. Since $s\in qT$ and $S$ contains all its factors, for some unit $u$ in $T$, $qu\in S$. Now since $S$ has completion $T$, $qT\cap S=quT\cap S=quS$.
\end{proof}

Next is a lemma about factorizations of nonzero non-units in $S^\star$-subrings of $T$.

\begin{lemma}\label{r=cd}
Let $(R,R\cap M)$ be an $S^\star$-subring of $T$ and let $r\in R\cap M$. Then there exists $c\in (x_1,\ldots,x_n)S=S\cap M$ and $d$ is a unit in $T$ satisfying $r=cd$.
\end{lemma}
\begin{proof}
First, if $r=0$, then $c=0$ and $d=1$ satisfy the desired conditions. Next, suppose $r\neq 0$. 
Since $r\in R\cap M\subseteq M$, $r$ is not a unit in $T$, so we can factor $r$ into primes in $T$ to obtain $r = p_1\cdots p_k$ for primes $p_i$ in $T$. For every $p_i$, $r\in p_iT\cap R$ and so by Lemma \ref{qTcapS=quS}, for some unit $u_i$ in $T$, $p_iu_i\in S$. Let $c=p_1u_1\cdots p_ku_k$ and $d=u_1^{-1}\cdots u_k^{-1}$. Then $c\in (x_1,\ldots, x_n)T\cap S=(x_1,\ldots, x_n)S$ and $d$ is a unit in $T$ such that $r=cd$.
\end{proof}

\section{The Construction} \label{construction of B}

In the following lemmas and propositions, we construct an ascending chain of $CS^\star$-subrings of $T$, starting from $S$, by adjoining power series, $u$, of the form specified below. In particular, upon adjoining each $u$, we not only retain the properties of $CS^\star$-subrings of $T$ but add new elements from $T$. The uncountable union of these rings, $A$, will be an uncountable $S^\star$-subring of $T$. To $A$, we adjoin elements so that the resulting $S^\star$-subring, $B$, of $T$ satisfies $bT\cap B=bB$ for any $b\in B$. Using this property of $B$ we show that $B$ is an excellent, uncountable, RLR with a countable spectrum. 
 Several of the definitions and lemmas in this section were inspired by ideas in W. Zhu's senior thesis at Williams College and work by C. Colbert in \cite{Colbert}.

Consider $u$ of the form
\begin{equation}\label{form}
u = 1 + A_1z_1 + A_2z_1z_2 + \cdots + A_kz_1z_2\cdots z_k+\cdots, 
\end{equation} 
where $A_i \in (x_1,\ldots, x_n)R$ and $z_i \in (x_1,\ldots, x_n)S$, for some $CS^\star$-subring, $R$, of $T$. Notice that, since $A_i,z_i\in M$, the $k$th term in the series is in $M^{k}$ for $k\geq 2$, so this series converges in $T$. So $u$ is, in fact, a unit in $T$. Now define,
 $$M_k = 1 + A_1z_1 + \cdots + A_{k-1}z_1\cdots z_{k-1}\text{ and }K_{k}= A_{k}z_{1}\cdots z_{k-1} + A_{k+1}z_{1}\cdots z_{k-1}z_{k+1}+ \cdots. $$ We can then express $u$ as $ M_{k} + z_kK_{k}$ for any $k\geq 1$. 
The next lemma demonstrates how we adjoin such an element $u$ to a $CS^\star$-subring, $R$, of $T$ so that the resulting ring is a $CS^\star$-subring of $T$. We do this by algorithmically choosing the values of $z_{i}$.

\begin{lemma}\label{R_n is a CS^star-subring}
Let $R$ be a $CS^\star$-subring of $T$. Then for any $A_i\in (x_1,\ldots, x_n)R$ satisfying the property that whenever $i>j$ there exists $k$ such that $A_i\in M^{k}$ and $A_j\notin M^{k}$, there exist $z_i \in (x_1,\ldots, x_n)S$ such that, if $u \in T$ is of the form,
$$u= 1 + A_1z_1 + A_2z_1z_2 + \cdots + A_kz_1z_2\cdots z_k+\cdots, $$
then $R[u]_{R[u]\cap M}$ is a $CS^\star$-subring of $T$.
\end{lemma}
 
\begin{proof}
First we define the $z_i$'s and then show that the resulting ring is a $CS^\star$-subring of $T$. Define $R''=R[X]$ where $X$ is an indeterminate. Since $R$ is a $CS^\star$-subring of $T$, $R$ is countable. Then $R''$, polynomials in $X$ over $R$, is also countable. Given this, we can enumerate the nonzero elements of $R''$ using the nonnegative integers. Consider the $i$th element in the well-order, $G_i(X)$. Substituting any $u$ of the form in equation \ref{form}, we have, for any $k\geq 1$,
\[G_i(u)=r_{\ell,i}u^\ell+\cdots +r_{1,i}u+r_{0,i}=r_{k,i}(M_{k}+z_kK_{k})^\ell+\cdots +r_{1,i}(M_{k}+z_kK_{k})+r_{0,i}\]
for $r_{m,i}\in R$. By binomial expansion, this becomes 
\begin{equation} \label{g_i}
G_{i}(u) = G_{i}(M_{k}) + z_{k}\left(\sum_{m= 1}^\ell r_{m,i}\sum_{j=1}^m\binom{m}{j}z_{k}^{j-1}K_{k}^j M_{k}^{m-j}\right).
\end{equation}

We will define the $z_i$'s recursively using $G_j(M_i)$ so that $z_i\in (x_1,\ldots, x_n)S$. Notice that, since each $z_i$ will be in $(x_1,\ldots, x_n)S$, the second term in (\ref{g_i}) is an element of $(x_1,\ldots, x_n)T$. This means that $G_i(u)\in (x_1,\ldots, x_n)T\iff G_i(M_{k})\in (x_1,\ldots, x_n)T$. Equivalently, $G_i(u)$ is a unit in $T$ if and only if $G_i(M_k)$ is a unit in $T$ for all $k\geq 1$. Since $A_i\in (x_1,\ldots, x_n)R$, we have that $M_i\in R$ and so $G_j(M_i)\in R$ for all $i,j\in \N$. 

Starting with $j=1$ and $i=1$, we will use $G_j(M_i)$ to define $z_i$ as follows:
\begin{enumerate}
\item If $G_j(M_i)$ is a unit, then let $z_{i}=x_1$ and use to $G_{j+1}(M_{i+1})$ to define $z_{i+1}$.
\item If $G_{j}(M_i)=0$, then let $z_{i}=x_1$ and use to $G_{j}(M_{i+1})$ to define $z_{i+1}$. 
\item If $G_j(M_i)$ is nonzero and not a unit, then since $G_j(M_i)\in R$, by Lemma \ref{r=cd} $G_j(M_i)=cd$ where $c\in (x_1,\ldots, x_n)S$ and $d$ is a unit in $T$. Since $G_j(M_i)\neq 0$, $c\neq 0$. Let $z_{i}=c$ and use $G_{j+1}(M_{i+1})$ to define $z_{i+1}$. 
\end{enumerate}

One consequence of the above definition is that $z_i\neq 0$ for all $i$ and that $z_i\in (x_1,\ldots, x_n)S$ for all $i$. Now that $u$ has been defined by specifying each $z_i$, define $R'=R[u]$. We will now show that $R'_{R'\cap M}$ is a $CS^\star$-subring of $T$. Notice that $R'$ is countable, and $S\subseteq R\subseteq R'$. Let $P$ be a prime ideal of $T$ with $P\cap S=(0)$. We will show that $P\cap R'=(0)$ by showing that, for any nonzero $r\in R'$, $r\notin P$. First, since $r$ is nonzero, $r=G_i(u)$ for some $i$. If $G_i(u)$ a unit in $T$, then $G_i(u)\notin P$. 

Next, suppose that $G_i(u)$ is not a unit. Then $G_i(M_k)$ is not a unit for all $k\in \N$. We will show that every $M_i$ is distinct. 
Suppose not, and that $M_{j}=M_{i}$, for some $i\neq j$. Without loss of generality, assume $i>j$. Then 
\[M_{i}-M_{j}= A_jz_1\cdots z_j+A_{j+1}z_1\cdots z_{j+1}+\cdots +A_{i-1}z_1\cdots z_{i-1}=0\]
Since every $z_i$ is nonzero, % and we have no zero divisors since \Q is an integral domain so $T$ is an integral domain 
\[A_j+A_{j+1}z_{j+1}+\cdots +A_{i-1}z_{j+1}\cdots z_{i-1}=0\]
Now choose $k$ such that $A_{j+1}\in M^k$ but $A_j\notin M^k$. Then 
\[A_{j+1}z_{j+1}+\cdots +A_{i-1}z_{j+1}\cdots z_{i-1}=-A_j\notin M^k\]
However $A_{j+1}\in M^k$, so $A_\ell\in M^k$ for every  $\ell\geq j+1$, thus 
\[A_{j+1}+\cdots +A_{i-1}z_{j+2}\cdots z_{i-1}\in M^{k}\]
a contradiction. Thus $M_i\neq M_j$ whenever $i\neq j$ and all the $M_i$'s are distinct.

 Since $0\neq G_i(u)\in R[u]$, $G_i(u)$ has at most $\deg(G_i)$ roots in $R$ because $R$ is an integral domain. Each $M_i\in R$ and all the $M_i$'s are distinct, thus there exists some $k\in \N$ such that $G_i(M_k)\neq 0$. Thus by case (3) in the algorithm above,  $G_i(M_k)=z_{k}d$ where $d$ is a unit. Substituting into (\ref{g_i}), we have 
\begin{equation}\label{factor G_i(u)}
G_{i}(u) = z_{k}\left(d+\sum_{m= 1}^{\ell} r_{m,i}\sum_{j=1}^m\binom{m}{j}z_{k}^{j-1}K_{k}^j M_{k}^{m-j}\right).
\end{equation}

Notice that $z_{k}\in S$ and $z_{k}\neq 0$, so then $z_{k}\notin P$ since $P\cap S=(0)$. Furthermore, since $K_k\in M$ and $d$ is a unit, $d+\sum_{m= 1}^{k+1} r_{m,i}\sum_{j=1}^m\binom{m}{j}z_{k}^{j-1}K_{k}^j M_{k}^{m-j}$ is a unit and thus not in $P$. Since $P$ is prime, this means that $G_i(u)\notin P$. Thus we have shown that $P\cap R'=(0)$. Finally, localizing $R'$ at $R'\cap(x_1,\ldots, x_n)T$ yields a $CS^\star$-subring of $T$. 
\end{proof}

The next lemma shows that there exists a choice of $A_i\in (x_1,\ldots, x_n)R$ such that the ring, $R[u]_{R[u]\cap M}$, from Lemma \ref{R_n is a CS^star-subring} is not equal to $R$. We would like to thank C. Colbert for discussions that led to the main idea in this proof. 

\begin{lemma}\label{pick-u}
Given a $CS^\star$-subring, $R$, of $T$, there exists a $CS^\star$-subring, $R'$, of $T$, where $R\subsetneq  R'\subset T$. 
\begin{proof}
First, by Lemma \ref{R_n is a CS^star-subring}, for any $A_i\in (x_1,\ldots, x_n)R$ satisfying the property that whenever $i>j$ there exists $k\in\N$ such that $A_i\in M^k$ but $A_j\notin M^k$, there exists a $u\in T$ with $u=1+A_1z_1+A_2z_1z_2+\cdots$, with $z_i\in (x_1,\ldots,x_n)S$, such that $R[u]_{R[u]\cap M}$ is a $CS^\star$-subring of $T$. Let $A_i=x_1^{q(i)}$ where $q:\Z^+\rightarrow \Z^+$ is a strictly increasing function. Notice that $A_i\in (x_1,\ldots, x_n)R$ and that, for $i>j$, $A_i\in M^{q(i)}$ but $A_j\notin M^{q(i)}$. 
 We will show that there are uncountably many possible $u$'s. First, notice that by a diagonal argument there are uncountably many choices for the function $q$.  We will show that distinct choices for the function $q$ yield distinct $u$'s. Let $u_1=1+A_1z_1+\cdots $ and $u_2=1+B_1z'_1+\cdots$, where $A_i=x_1^{q(i)}$ and $B_i=x_1^{p(i)}$ where $p,q:\Z^+\rightarrow \Z^+$ are both strictly increasing. Suppose that $u_1=u_2$. We will show that $A_i=B_i$ and $z_i=z'_i$ for all $i$. 

First, notice that $M_1=1$. Since $z_1$ and $z'_1$ are both defined by the algorithm in the proof of Lemma \ref{R_n is a CS^star-subring} using $G_1(M_1)=G_1(1)$, they will be the same and nonzero. Then equating $u_1$ and $u_2$, we have $A_1+A_2z_2+\cdots=B_1+B_2z_2'+\cdots$, so then
\[x_1^{q(1)} +x_1^{q(2)}z_2+\cdots = x_1^{p(1)}+x_1^{p(2)}z_2'+\cdots  \]
Without loss of generality, suppose $q(1)\leq p(1)$, then 
\[1 +x_1^{q(2)-q(1)}z_2+\cdots = x_1^{p(1)-q(1)}+x_1^{p(2)-q(1)}z_2'+\cdots\]
Since $z_i\in (x_1,\ldots, x_n)S$ for all $i$ and $1\notin M$, the left hand side is not in $M$. Thus the right hand side is not in $M$ either, however all but the first term are in $M$, since $z'_i\in (x_1,\ldots, x_n)S$. Thus $x_1^{p(1)-q(1)}\notin M$, which implies $p(1)=q(1)$ and so $A_1=B_1$. Since the definition of $z_i$ (respectively $z_i'$) depends only on $A_j$ and $z_j$ (respectively $B_j$ and $z_j'$) for all $j<i$, and we have shown that $z_1=z_1'$ and $A_1=B_1$, we can show inductively that $z_i=z_i'$ and $A_i=B_i$ for all $i\geq 1$. Since there are uncountably many choices for the function $q$, there are uncountably many units $u$ such that $R[u]_{R[u]\cap M}$ is a $CS^\star$-subring of $T$. Since $R$ is a $CS^\star$-subring of $T$, it is countable, and thus there exists a $u\in T\backslash R$ such that $R[u]_{R[u]\cap M}$ is a $CS^\star$-subring of $T$. Hence $R'=R[u]_{R[u]\cap M}$ is the desired $CS^\star$-subring of $T$. 
\end{proof}
\end{lemma}

\begin{theorem}\label{ring A}
There exists an uncountable, quasi-local $S^\star$-subring of $T$, $(A,A\cap M)$.
\end{theorem}
\begin{proof}
There exists a well-ordered, uncountable set, $C$, such that every element of $C$ has only countably many predecessors. Let $0$ denote the minimal element of $C$. We will inductively define a $CS^\star$-subring of $T$, $S_c$, for every element $c\in C$. Define $S_0=S$. Let $0<c\in C$ and assume that $S_b$ has been defined for every $b<c$. If $c$ has a predecessor, $b\in C$, then define $S_c$ to be the $CS^\star$-subring obtained from Lemma \ref{pick-u} with $R=S_{b}$, so that $S_{b}\subsetneq S_c\subseteq T$. If $c$ is a limit ordinal, define $S_c=\cup_{b<c}S_b$. Then every $S_c$ is a $CS^\star$-subring of $T$. 

Let $A=\cup_{c\in C}S_c$. Since each $S_c$ is a $CS^\star$-subring of $T$, $A$ is a quasi-local $S^\star$ subring of $T$ with unique maximal ideal $A\cap M$.
Finally, notice that, for uncountably many $c\in C$, $c$ has a predecessor $b$ such that $S_b\subsetneq S_c$, so then $A$ is uncountable. Thus $(A,A\cap M)$ is an uncountable, quasi-local $S^\star$-subring of $T$. 
\end{proof}

Next, we adjoin elements to the ring $A$ such that, for the resulting ring, $B$, $bT\cap B=bB$ for all $b\in B$.

\begin{lemma}\label{IT cap A < IA'}
Given an uncountable $S^\star$-subring, $A$, of $T$, there exists an uncountable $S^\star$-subring, $A'$, of $T$, with $A'\supset A$, such that, for any principal ideal $I$ of $A$, $IT\cap A \subseteq IA'$. 
\end{lemma}
\begin{proof}
Consider the set $\Omega=\{(I,c):I\text{ is a principal ideal of }A, c\in IT\cap A\}$. Well-order $\Omega$ with $0$ denoting its minimal element. We will inductively define subrings for each $(I,c)\in\Omega$. Let $S_0=A$. Given $0<\alpha\in\Omega$, where $S_\beta$ has been defined to be a $CS^\star$-subring of $T$ for every $\beta<\alpha$, we will define $S_\alpha$ as follows. If $\alpha$ is a successor ordinal with predecessor $\beta=(I,c)$, then, since $I$ is a principal ideal of $A$, $I=aA$ for some $a\in A$. If $a=0$, then define $S_\alpha=S_\beta$. Otherwise, $c\in IT\cap A=aT\cap A$ implies that $c=at$ for some $t\in T$. Define $S_\alpha=S_\beta[t]$. Notice that $c\in IS_\alpha$. If $\alpha$ is a limit ordinal, define $S_\alpha=\cup_{\beta<\alpha}S_\beta$. 

 We will now show inductively that $S_\alpha$ is a $S^\star$-subring for every $\alpha\in\Omega$. Notice that $S_0=A$ is an $S^\star$-subring by assumption. Next, suppose that $S_\beta$ is a $S^\star$-subring for every $\beta<\alpha$. We will show that $S_\alpha$ is as well. By construction, $S\subseteq A\subseteq S_\alpha\subseteq T$ for every $\alpha\in\Omega$ so $S_\alpha$ is uncountable and $S\subseteq S_\alpha\subseteq T$. Thus we need only check that, for any $P\in\spec T$ such that $P\cap S=(0)$, we have $P\cap S_\alpha=(0)$.

If $\alpha$ is a successor ordinal with predecessor $\beta=(I,c)$, either $I=(0)$ or $I=aA$ for $a\neq 0$. In the case where $I=(0)$, $S_\alpha=S_\beta$ and so $S_\alpha$ is an $S^\star$-subring of $T$. In the case where $I=aA$ for some nonzero $a\in A$, then $S_\alpha=S_\beta[t]$ with $t\in T$ such that $c=at\in IT=aT$. Consider $P\in \spec T$ such that $P\cap S=(0)$. Let $g\in P\cap S_\alpha$. Then $g=r_kt^k+\cdots+r_1t+r_0$ for $r_i\in S_\beta$. Then $a^kg\in P\cap S_\beta =(0)$. Since $a\neq 0$, this implies that $g=0$ and so $P\cap S_\alpha=(0)$ as desired. 
 Next, if $\alpha$ is a limit ordinal, then $S_\alpha=\cup_{\beta<\alpha}S_\beta$. Since each $S_\beta$ is an $S^\star$-subring of $T$, so is $S_\alpha$. Thus $S_\alpha$ is an $S^\star$-subring of $T$ for every $\alpha\in\Omega$. 
 
 Now define $A'=\cup_{\alpha\in\Omega}S_\alpha$. We will show that this is the desired ring. First, note that $S\subset A\subset A'\subset T$. Since $A$ is uncountable, so is $A'$. Next, since each $S_\alpha$ is an $S^\star$-subring of $T$, so is $A'$. Finally, consider some principal ideal $I$ of $A$. Let $c\in IT\cap A$. Then $\alpha=(I,c)\in \Omega$ and so $c\in IS_{\alpha+1}\subseteq IA'$. Thus $IT\cap A\subseteq IA'$ as desired. 
\end{proof}

\begin{theorem}\label{ring B}
There exists an uncountable, quasi-local $S^\star$-subring, $(B,B\cap M)$, such that, for every principal ideal $I$ of $B$, $IT\cap B=IB$. 
\end{theorem}
\begin{proof}
Beginning with $B_0=A$, the ring obtained from Theorem \ref{ring A}, inductively define $B_{i+1}$ as the ring obtained from Lemma \ref{IT cap A < IA'}, such that, for any principal ideal $I$ of $B_i$, $IT\cap B_i\subseteq IB_{i+1}$. Then let $B'=\cup_{i=0}^\infty B_i$. As each $B_i$ is an uncountable $S^\star$-subring, so is $B'$. Next, consider some principal ideal $I$ of $B'$. We have that $I=bB$ for some $b\in B'$. Given $c\in bT\cap B'$, $c\in B'$, so for some $i$, we have that $b,c\in B_i$. By construction, $c\in bT\cap B_i\subseteq bB_{i+1}\subseteq IB'$, so then $IT\cap B'\subseteq IB'$. The other direction follows trivially, and so we have shown that $IT\cap B'=IB'$. Now define $B=B'_{B'\cap M}$. By the same argument as in the proof of Lemma \ref{close up ideals}, for every finitely generated ideal $I$ of $B$, $IT\cap B=IB$. Thus $(B,B\cap M)$ is the desired uncountable, quasi-local $S^\star$-subring of $T$. 
\end{proof}

\section{Properties of the Final Ring}\label{properties of B}

In this section, $B$ refers to the ring constructed in Theorem \ref{ring B} of the previous section. Recall that $(B,B\cap M)$ is an uncountable, quasi-local $S^\star$-subring of $T$ such that, for every principal ideal $I$ of $B$, $IT\cap B=IB$. In this section we will use the properties of $B$ to show that $B$ is Noetherian and has completion $T$, and then demonstrate that $B$ is excellent and a RLR.

\begin{lemma}\label{extended ideals}
Finitely generated ideals of $B$ are extended from $S$, i.e. for any finitely generated ideal $I$ of $B$,  $I=(p_1,\ldots, p_k)B$ for $p_i\in S$. 
\begin{proof}
Let $I$ be a finitely generated ideal of $B$. Then $I=(b_1,\ldots, b_k)B$ for $b_i\in B$. If $I=B$, then $I=1B$ and $1\in S$.
Thus suppose that $I$ is a proper  ideal and so the $b_i$'s are non-units. 

Since the $b_i$'s are not units, $b_i\in B\cap (x_1,\ldots, x_n)T$. By Lemma \ref{r=cd}, $b_i=p_iu_i$ where $p_i\in (x_1,\ldots, x_n)S$ and $u_i$ is a unit in $T$ for all $i=1,2,\ldots, k$. We will show that $u_i$ is also a unit in $B$. Notice that $b_i=p_iu_i\in p_iT\cap B$, and since $p_iT$ is a principal ideal, $p_iT\cap B=p_iB$. Hence $p_iu_i\in p_iB$, which shows that $u_i\in B$. Since the $p_i$'s and $b_i$'s are associates in $B$, $ I=(b_1,\ldots, b_k)B =(p_1,\ldots, p_k)B$, where $p_i\in S$, and so $I$ is extended from $S$.
\end{proof}
\end{lemma}

\begin{theorem}\label{closed ideals of ring B}
For every finitely generated ideal $I$ of $B$, $IT\cap B=IB$. 
\end{theorem}
\begin{proof}
Consider some finitely generated ideal $I$ of $B$. We know by Lemma \ref{extended ideals} that $I=(p_1,\ldots, p_k)B$ for $p_i\in  S$. Let $c\in IT\cap B$. We will show that $c\in IB$. If $I=B$, then $IT\cap B=BT\cap B=B=IB$ as desired. Otherwise $IT\subseteq M$, and so $c\in IT\cap B\subseteq M\cap B$.
%% each p_i\in B\cap M\subset M since I\neq B, so then IT\subset MT=M
 Then, by Lemma \ref{r=cd}, $c=qu$ for $q\in (x_1,\ldots, x_n)S$ and unit $u$ in $T$. Thus $qu=c\in qT\cap B=qB$ since $qB$ is a principal ideal of $B$. Since $qu\in qB$, $u\in B$ and since $u\notin M$, $u\notin B\cap M$, and thus $u$ is a unit in $B$. 
 Observe,
\[cu^{-1}=q\in (p_1,\ldots,p_k)T\cap S=(p_1,\ldots, p_k)S\subset (p_1,\ldots,p_k)B=IB.\]
Since $u$ is a unit in $B$, this implies that $c\in IB$ as well, as desired. 
\end{proof}

As a consequence of Theorem \ref{closed ideals of ring B}, $B\cap M=(x_1,\ldots, x_n)B$, and so the maximal ideal of $B$ is $(x_1,\ldots, x_n)B$.

\begin{theorem}\label{completion of B}
The ring $(B,B\cap M)$ is Noetherian with completion $T$ and is a RLR.   
\end{theorem}
\begin{proof}
We have that $(B,B\cap M)$ is a quasi-local subring of the complete local ring $T$ with $R_0\subset B$. By Theorem \ref{closed ideals of ring B}, for every finitely generated ideal $I$ of $B$, $IT\cap B=IB$. Thus by Corollary \ref{alt completion proving machine}, $B$ is Noetherian with completion $T$ and $B$ is a RLR. 
\end{proof}

In our final theorem we will show that $B$ is excellent with a countable spectrum.

\begin{theorem}\label{B is excellent}
For any $n\geq 2$, there exists an uncountable $n$-dimensional, excellent, regular local ring, with countable spectrum. 
\end{theorem}
\begin{proof}
We will show that $B$ is the desired ring. 
We already have that $B$ is an uncountable RLR with $\widehat{B}=T$, and so $\dim(B)=\dim(T)=n$. Thus, all that remains to be shown is that $B$ is excellent with a countable spectrum. Since $(B,B\cap M)$ is local with $\widehat{B}=T$, by Lemma \ref{alt def of excellent}, to show that $B$ is excellent it is sufficient to show that, for every $P\in\spec B$ and $Q\in \spec T$ where $Q\cap B=P$, $(T/PT)_{Q}$ is a RLR. 
Let $P\in \spec B$ and $Q\in\spec T$ such that $Q\cap B= P$. Notice that 
\[Q\cap S = (Q\cap B)\cap S = P\cap S.\]
Since $B$ is Noetherian, every ideal is finitely generated. Furthermore, by Lemma \ref{extended ideals}, every finitely generated ideal of $B$ is extended from $S$, that is, it can be generated by elements in $S$. Let $P=(p_1,\ldots,p_k)B$ for $p_i\in S$. 
Then 
\[P\cap S = (p_1,\ldots, p_k)B\cap S \subseteq (p_1,\ldots, p_k)T\cap S = (p_1,\ldots,p_k)S\]
where the last equality follows because $IT\cap S=IS$, for every finitely generated ideal $I$ of $S$. Clearly, $(p_1,\ldots, p_k)S\subseteq P\cap S$, so then we have that $(p_1,\ldots,p_k)B\cap S=(p_1,\ldots, p_k)S$, and thus $Q\cap S=(p_1,\ldots, p_k)S$. 

Since $S$ is excellent with completion $T$, we know that $(T/(p_1,\ldots,p_k)T)_{Q}$ is a RLR. Since $(p_1,\ldots, p_k)T = PT$, we have that $(T/(p_1,\ldots,p_k)T)_{Q}=(T/PT)_{Q}$ and so $(T/PT)_{Q}$ is a RLR, as desired. Thus by Lemma \ref{alt def of excellent}, $B$ is excellent.

Finally, we will prove that $\spec B$ is countable. 
Let $I_R$ be the set of ideals of a ring $R$. We will show first that $I_B$ is countable. Define $f:I_S\rightarrow I_B$ by $(a_1,\ldots, a_k)S\mapsto (a_1,\ldots, a_k)B$. Since $S$ is Noetherian this function is well-defined. We will show that $f$ is surjective. Let $I$ be an ideal of $B$. By Lemma \ref{extended ideals}, $I=(p_1,\ldots, p_k)B$ for $p_i\in S$. Then $J=(p_1,\ldots, p_k)S$ is an ideal of $S$, and $f(J)=I$, so $f$ is surjective. Note that $I_S$ is countable since $S$ is Noetherian and countable. Thus $I_B$ is countable. Since $\spec B\subseteq I_B$, $B$ has a countable spectrum.  
\end{proof}

\section*{Acknowledgments}
The authors would like to thank Weitao Zhu whose senior thesis at Williams College inspired many of the definitions and constructions in this paper. We would also like to thank Cory Colbert whose work inspired our question and who contributed key ideas for the proof of Lemma \ref{pick-u}.  
Some of this work was completed during the SMALL REU at Williams College supported by funding from both an NSF grant (DMS-1659037) and the Clare Boothe Luce Scholarship Program.

%REFERENCES 
\bibliographystyle{unsrt}
\bibliography{refs}
\end{document}